\pdfoutput=1
\documentclass{amsart}
\usepackage{amsmath,
 amssymb,
 fullpage,
 mathtools,
 mathpazo,
 thmtools,
 enumerate
}
\usepackage[
 pagebackref,
 colorlinks,
 allcolors=blue
]{hyperref}

\usepackage[capitalize]{cleveref}

\newlength{\bibitemsep}
\newlength{\bibparskip}
\let\oldthebibliography\thebibliography
\renewcommand\thebibliography[1]{%
 \oldthebibliography{#1}%
 \setlength{\parskip}{\bibparskip}%
 \setlength{\itemsep}{\bibitemsep}%
}

\newcommand{\ZZ}{\mathbf{Z}}
\newcommand{\QQ}{\mathbf{Q}}
\newcommand{\Qp}{\mathbf{Q}_p}
\newcommand{\QQbar}{\overline{\QQ}}
\newcommand{\Zp}{\ZZ_p}

\newcommand{\FF}{\mathbf{F}}

\newcommand{\rb}{\bar{\rho}}
\newcommand{\gf}{\mathfrak{g}_{\FF}}

\newcommand{\cE}{\mathcal{E}}

\newcommand{\cG}{\mathcal{G}}
\newcommand{\cL}{\mathcal{L}}
\newcommand{\cO}{\mathcal{O}}
\newcommand{\cR}{\mathcal{R}}
\newcommand{\cS}{\mathcal{S}}
\newcommand{\cT}{\mathcal{T}}
\newcommand{\cV}{\mathcal{V}}

\newcommand{\fX}{\mathfrak{X}}
\newcommand{\ord}{\mathrm{ord}}
\newcommand{\nord}{\mathrm{no}}
\newcommand{\der}{\mathrm{der}}
\newcommand{\htimes}{\operatorname{\hat{\otimes}}}

\newcommand{\bV}{\overline{V}}

\renewcommand{\le}{\leqslant}

\renewcommand{\ge}{\geqslant}

\DeclareMathOperator{\Spf}{Spf}

\DeclareMathOperator{\Gal}{Gal}
\DeclareMathOperator{\Frob}{Frob}
\DeclareMathOperator{\GL}{GL}

\DeclareMathOperator{\Ad}{Ad}

\newcommand{\CNLO}{\underline{\mathrm{CNL}}_{\cO}}

\newtheorem{definition}{Definition}[section]
\newtheorem{theorem}[definition]{Theorem}
\newtheorem{proposition}[definition]{Proposition}

\newtheorem{conjecture}[definition]{Conjecture}

\theoremstyle{remark}
\declaretheorem[name=Remark,sibling=theorem,qed={\lower-0.3ex\hbox{$\diamond$}}]{remark}
\declaretheorem[name=Example,sibling=theorem,qed={\lower-0.3ex\hbox{$\diamond$}}]{example}

\begin{document}
 \title{P-adic L-functions in universal deformation families}
  \author{David Loeffler}
 \address{David Loeffler\\
  Mathematics Institute, University of Warwick, Coventry CV4 7AL, UK}
 \email{d.a.loeffler@warwick.ac.uk}
 \urladdr{\url{http://warwick.ac.uk/fac/sci/maths/people/staff/david_loeffler/}}
 \thanks{Supported by a Royal Society University Research Fellowship.}

 \renewcommand{\emailaddrname}{\emph{Email}}
 \renewcommand{\curraddrname}{\emph{ORCID}}
 \curraddr{\href{http://orcid.org/0000-0001-9069-1877}{\texttt{orcid.org/0000-0001-9069-1877}}}

 \begin{abstract}
  We construct examples of $p$-adic $L$-functions over universal deformation spaces for $\GL_2$. We formulate a conjecture predicting that the natural parameter spaces for $p$-adic $L$-functions are not the usual eigenvarieties (parametrising nearly-ordinary families of automorphic representations), but other, larger spaces depending on a choice of a parabolic subgroup, which we call `big parabolic eigenvarieties'.
 \end{abstract}

 \maketitle

\section{Introduction}

 It is well known that many interesting automorphic $L$-functions $L(\pi, s)$ have $p$-adic counterparts; and that these can often be extended to multi-variable $p$-adic $L$-functions, in which the automorphic representation $\pi$ itself also varies in a $p$-adic family of some kind. In the literature so far, the $p$-adic families considered have been \emph{Hida families}, or more generally \emph{Coleman families} -- families of automorphic representations which are principal series at $p$, together with the additional data of a ``$p$-refinement'' (a choice of one among the Weyl-group orbit of characters from which $\pi_p$ is induced). In Galois-theoretic terms, this corresponds to a full flag of subspaces in the local Galois representation at $p$ (or in its $(\varphi, \Gamma)$-module, for Coleman families). The parameter spaces for these families are known as \emph{eigenvarieties}.

 The aim of this note is to give an example of a $p$-adic $L$-function varying in a family of a rather different type: it arises from a family of automorphic representations of $\GL_2\times \GL_2$, but the parameter space for this family (arising from Galois deformation theory) has strictly bigger dimension than the eigenvariety for this group -- it has dimension 4, while the eigenvariety in this case has dimension 3. We also sketch some generalisations of the result which can be proved by the same methods. This corresponds to the fact that a $p$-refinement is a little more data than is actually needed to define a $p$-adic $L$-function: rather than a full flag, it suffices to have a single local subrepresentation of a specific dimension (a Panchishkin subrepresentation), which is a weaker condition and hence permits variation over a larger parameter space.

 We conclude with some speculative conjectures whose aim is to identify the largest parameter spaces on which $p$-adic $L$-functions and Euler systems can make sense. We conjecture that, given a reductive group $G$ and parabolic subgroup $P$ (and appropriate auxiliary data), there should be two natural $p$-adic formal schemes, the \emph{big} and \emph{small} $P$-nearly-ordinary eigenvarieties. These coincide if $P$ is a Borel subgroup, but not otherwise; if $G = \GL_2$ and $P$ is the whole of $G$, then the big eigenvariety is the 3-dimensional Galois deformation space of a modular mod $p$ representation (with no local conditions at $p$). In general, we expect that the ``natural home'' of $p$-adic $L$-functions -- and also of Euler systems -- should be a big ordinary eigenvariety for an appropriate parabolic subgroup.

 \subsubsection*{Acknowledgements} It is a pleasure to dedicate this article to Bernadette Perrin-Riou, in honour of her immense and varied contributions to number theory in general, and to $p$-adic $L$-functions in particular, which have been an inspiration to me throughout my career.

 I would also like to thank Daniel Barrera Salazar, Yiwen Ding, and Chris Williams for informative discussions in connection with this paper, and Sarah Zerbes for her feedback on an earlier draft.

\section{Families of Galois representations}

 \subsection{The Panchishkin condition}

  Let $L$ be a finite extension of $\Qp$ and let $V$ be a finite-dimensional vector space with a continuous linear action of $\Gamma_{\QQ} = \Gal(\overline{\QQ}/\QQ)$. Recall that $V$ is said to be \emph{geometric} if it is unramified at all but finitely many primes and de Rham at $p$; in particular it is Hodge--Tate at $p$, so we may consider its Hodge--Tate weights. (In this paper, we adopt the common, but not entirely universal, convention that the cyclotomic character has Hodge--Tate weight $+1$.)

  \begin{definition}
   \label{def:panch}
   We say $V$ satisfies the \textbf{Panchishkin condition} if it is geometric, and the following conditions hold:
   \begin{enumerate}
    \item We have
    \[ (\text{number of Hodge--Tate weights $\ge 1$ of $V$}) = \dim V^{(c = +1)}\]
    where $c \in \Gamma_\QQ$ is (any) complex conjugation.
    \item There exists a subspace $V^+ \subseteq V$ stable under $\Gamma_{\QQ_p}$ such that $V^+$ has all Hodge--Tate weights $\ge 1$, and $V / V^+$ has all Hodge--Tate weights $\le 0$.
   \end{enumerate}
  \end{definition}

  \begin{remark} \
   \begin{enumerate}[(i)]

    \item Note that $V^+$ is unique if exists; we call it the \emph{Panchishkin subrepresentation} of $V$ at $p$.

    \item If $V$ is the $p$-adic realisation of a motive $M$, then condition (1) is equivalent to requiring that $L(M, 0)$ is a critical value of the $L$-function $L(M, s)$ in the sense of \cite{deligne79}.

    \item The Panchishkin condition is closely related to the concept of \emph{near-ordinarity}: a representation $V$ is said to be \emph{nearly-ordinary} if it is geometric, and there exists a full flag of subspaces of $V$ such that the Hodge--Tate weights of the graded pieces are in increasing order. However, we want to emphasise here that near-ordinarity is an unnecessarily restrictive hypothesis for the study of $p$-adic $L$-functions.\qedhere
   \end{enumerate}
  \end{remark}

 \subsection{Panchishkin families}

  By a ``Panchishkin family'', we mean a family of $p$-adic Galois representations \emph{equipped with a family of Panchishkin subobjects}. For simplicity, we shall suppose here that $p > 2$, so that the action of complex conjugation is diagonalisable. Let $\cO$ be the ring of integers of $L$, and $\FF$ its residue field. We let $\CNLO$ be the category of complete Noetherian local $\cO$-algebras with residue field $\FF$.

  \begin{definition}
   \label{def:panchfam}
   Let $\cR$ be an object of $\CNLO$. A \emph{Panchishkin family of Galois representations} over $\cR$ consists of the following data:
   \begin{itemize}
    \item a finite free $\cR$-module $\cV$ with an $R$-linear continuous action of $\Gamma_{\QQ}$, unramified at almost all primes.
    \item an $\cR$-direct-summand $\cV^+ \subseteq \cV$ stable under $\Gamma_{\Qp}$, of $\cR$-rank equal to that of $\cV^{c = 1}$,
   \end{itemize}
   satisfying the following condition:
   \begin{itemize}
    \item The set $\Sigma(\cV, \cV^+)$ of maximal ideals $x$ of $R[1/p]$ such that $\cV_x$ satisfies the Panchishkin condition and $\cV^+_x$ is its Panchishkin subrepresentation is dense in $\operatorname{Spec} R[1/p]$.
   \end{itemize}
  \end{definition}

  \begin{example}[Cyclotomic twists of a fixed representation]
   The original examples of Panchishkin families are those of the following form. Let $V$ be an $L$-linear representation of $\Gamma_{\QQ}$ satisfying the Panchishkin condition, and $V^\circ$ a $\cO$-lattice in $V$ stable under   $\Gamma_{\QQ}$. We let $\Lambda$ denote the Iwasawa algebra $\cO[[\Zp^\times]]$, and $\mathbf{j}$ the canonical character $\ZZ_p^\times \to \Lambda^\times$.

   If $\dim V^{c=1} = \dim V^{c=-1}$, then we can take $\cR$ to be the localisation of $\Lambda$ at any of its $(p-1)$ maximal ideals, corresponding to characters $\Zp^\times \to \FF^\times$; otherwise, we need to assume our maximal ideal corresponds to a character trivial on $-1$. We can then let $\cV = V^\circ \otimes (\chi_{\mathrm{cyc}})^{\mathbf{j}}$, and $\cV^+ = V^{\circ+} \otimes (\chi_{\mathrm{cyc}})^{\mathbf{j}}$, where $V^{\circ+} = V^\circ \cap V^+$.

   By construction, $\Sigma(\cV, \cV^+)$ contains all points of $\operatorname{Spec} \mathcal{R}[1/p]$ corresponding to characters of the form\footnote{We use additive notation for characters, so $j + \chi$ is a shorthand for the character $z \mapsto z^j \chi(z)$.} $j + \chi$, where $\chi$ is of finite order and $j$ is an integer in some interval containing 0 (depending on the gap between the Hodge--Tate weights of $V^+$ and $V/V^+$). In particular, it is Zariski-dense, as required.
  \end{example}

  The following conjecture is due to Coates--Perrin-Riou \cite{coatesperrinriou89} and Panchishkin \cite{panchishkin94} in the case of cyclotomic twists of a fixed representation. The generalisation to families as above is ``folklore''; we have been unable to locate its first appearance, but is (for instance) a special case of more general conjectures of Fukaya and Kato \cite{fukayakato06} (who have also investigated the case of non-commutative base rings $\cR$, which we shall not attempt to consider here).

  \begin{conjecture}
   \label{conj:main}
   For $(\cV, \cV^+)$ as above, there exists an element $\cL(\cV, \cV^+) \in \operatorname{Frac} \cR$ such that for all $x \in \Sigma(\cV, \cV^+)$ we have
   \[ \mathcal{L}(\cV, \cV^+)(x) = (\text{Euler factor}) \cdot \frac{L(M_x, 0)}{(\text{period})},\]
   where $M_x$ is the motive whose realisation is $\cV_x$.
  \end{conjecture}

  If $\cV_x$ is semistable at $p$, the expected form of the Euler factor is
  \[ \det\left[ (1 - \varphi)  : \mathbf{D}_{\mathrm{cris}}(V^+)\right] \cdot \det\left[ (1 - p^{-1}\varphi^{-1}): \mathbf{D}_{\mathrm{cris}}(V/V^+)\right]. \]
  We refer to \cite{fukayakato06} for more details of the interpolation factors involved.

 \subsection{Euler systems}
  \label{sect:euler}

  In \cite{LZ20-localconds}, Zerbes and the present author formulate a slightly more general version of the Panchishkin condition, depending on an integer $r$ with $0 \le r \le \dim V^{c = 1}$, which we call the ``$r$-Panchishkin condition''; the usual Panchishkin condition is the case $r = 0$. The definitions of the previous section extend naturally to give a notion of an \emph{$r$-Panchishkin family} $(\cV, \cV^+)$.

  We conjectured in \emph{op.cit.}~that when $\cV$ is the family of cyclotomic twists of a fixed representation, the $r$-Panchishkin condition was the ``correct'' condition for a family of Euler systems of rank $r$ to exist, taking values in the Galois cohomology of the Tate dual $\cV^*(1)$ and satisfying a local condition at $p$ determined by $\cV^+$. This extends the conjectures formulated by Perrin-Riou in \cite{perrinriou95}, which correspond to taking $r$ to be the maximal value $\dim V^{c = 1}$ (in which case $\{0\}$ is a Panchishkin subrepresentation). It is also consistent with the above conjectures of Coates--Perrin-Riou and Panchishkin for $r = 0$, if we understand a ``rank 0 Euler system'' to be a $p$-adic $L$-function.

  It seems natural to expect that an analogoue of Conjecture \ref{conj:main} should hold for arbitrary $r$-Panchishkin families; and, as in the rank 0 case, one can show that this would follow as a consequence of the very general conjectures of \cite{fukayakato06}.

  \begin{remark}
   There are a number of (unconditional) results concerning the variation of Euler systems for families of Galois representations arising from Hida families of automorphic representations, which are examples of nearly-ordinary families; see e.g.~\cite{ochiai05} for Kato's Euler system, and \cite{LLZ14} for the $\GL_2 \times \GL_2$ Beilinson--Flach Euler system.

   However, the above conjecture predicts that Euler systems should vary in more general families, which are not nearly-ordinary but are still $r$-Panchishkin. Some examples of cyclotomic twist type for $r = 1$ are discussed in \cite{LZ20-localconds}. A much more sophisticated example due to Nakamura, in which $\cR$ is the universal deformation space of a 2-dimensional modular Galois representation, is discussed in \S \ref{sect:nakamura} below.
  \end{remark}

 \section{Examples from \texorpdfstring{$\GL_2$}{GL(2)}}

 \subsection{The universal deformation ring}

 Let $\rb: \Gamma_{\QQ} \to \GL(\bV) \cong \GL_2(\FF)$ be a 2-dimensional, odd, irreducible (hence, by Khare--Wintenberger, modular) representation. We shall assume $\rb$ satisfies the following:
 \begin{itemize}
  \item $\rb|_{\Gamma_K}$ is irreducible, where $K = \QQ(\zeta_p)$ (Taylor--Wiles condition).
  \item if $\rb|_{\Gamma_{\QQ_p}}$ is not absolutely irreducible, with semisimplification $\chi_{1, p} \oplus \chi_{2, p}$ (after possibly extending $\mathbf{F}$), then we have $\chi_{1,p} / \chi_{2, p} \notin \{ 1, \varepsilon_p^{\pm 1}\}$ where $\varepsilon_p$ is the mod $p$ cyclotomic character.
  \item $\rb$ is unramified away from $p$.
 \end{itemize}

 Note that the first two assumptions are essential to our method (because they are hypotheses for major theorems which we need to quote). On the other hand, the third is imposed solely for convenience and could almost certainly be dispensed with.

 \begin{definition}
  Let $\cR(\rb) \in \CNLO$ be the universal deformation ring over $\cO$ parametrising deformations of $\rb$ as a $\Gamma_{\QQ, \{p\}}$-representation, and $\rho: \Gamma_{\QQ, \{p\}} \to \GL_2(\cR(\rb))$ the universal deformation. Let $\fX(\rb) = \Spf \cR(\rb)$.
 \end{definition}

 \begin{theorem}[B\"ockle, Emerton] \
  \label{thm:BE}
  \begin{itemize}
   \item The ring $\cR(\rb)$ is a reduced complete intersection ring, and is flat over $\cO$ of relative dimension 3.

   \item We have a canonical isomorphism $\cR(\rb) \cong \cT(\rb)$, where $\cT(\rb)$ is the localisation at the maximal ideal corresponding to $\rb$ of the prime-to-$p$ Hecke algebra acting on the space $\cS(1, \cO)$ of cuspidal $p$-adic modular forms of tame level 1.
  \end{itemize}
 \end{theorem}

 \begin{proof}
  This is proved in \cite{boeckle01} assuming that $\rb|_{\Gamma_{\Qp}}$ has a twist which is either ordinary, or irreducible and flat. This was extended to the setting described above (allowing irreducible but non-flat $\rb$) by Emerton, see \cite[Theorem 1.2.3]{emerton-localglobal}.
 \end{proof}

 \begin{remark}
  If $\rb$ is \emph{unobstructed} in the sense that $H^2\left(\Gamma_{\QQ, \{p\}},\Ad(\rb)\right) = 0$, then $\cR(\rb)$ is isomorphic to a power-series ring in 3 variables over $\cO$. It is shown in \cite{weston04} that if $f$ is a fixed newform of weight $\ge 3$, then for all but finitely many primes $\mathfrak{p}$ of the coefficient field $\QQ(f)$, the mod $\mathfrak{p}$ representation $\rb_{f, \mathfrak{p}}$ is unobstructed.
 \end{remark}

 \begin{definition} \
  \begin{enumerate}
   \item[(i)] If $f$ is a classical modular newform of $p$-power level (and any weight) such that $\rb_{f, p} = \rb$, then $\rho_{f, p}$ is a deformation of $\rb$ and hence determines a $\QQbar_p$-point of $\fX(\rb)$. We shall call these points \emph{classical}.
   \item[(ii)] More generally, a $\QQbar_p$-point of $\fX(\rb)$ will be called \emph{nearly classical} if the corresponding Galois representation $\rho$ has the form $\rho_{f, p} \otimes (\chi_{\mathrm{cyc}})^{-t}$, for some (necessarily unique) newform $f$ and $t \in \ZZ$.
  \end{enumerate}
 \end{definition}

 In the setting of (ii), if $t \ge 0$, the Galois representation $\rho_{f, p} \otimes (\chi_{\mathrm{cyc}})^{-t}$ corresponds formally to the nearly-overconvergent $p$-adic modular form $\theta^t(f)$, where $\theta = q \frac{\mathrm{d}}{\mathrm{d}q}$ is the Serre--Tate differential operator on $p$-adic modular forms. Slightly abusively, we denote such a point by $\theta^t(f)$, even if $t < 0$ (in which case $\theta^t(f)$ may not actually exist as a $p$-adic modular form).

 Theorem 1.2.4 of \cite{emerton-localglobal}, combined with Theorem 0.4 of \cite{pillonistroh16} in the case of equal Hodge--Tate weights, shows that any $\QQbar_p$-point $\rho$ of $\fX(\rb)$ which is de Rham at $p$ is a nearly-classical point (as predicted by the Fontaine--Mazur conjecture).

 \begin{proposition}
  For any weight $k \ge 2$, modular points corresponding to weight $k$ modular forms are dense in $\fX(\rb)$.
 \end{proposition}

 \begin{proof}
  This is obvious for $\Spf \cT(\rb)$, since $\cT(\rb)$ can be written as an inverse limit of localisations of Hecke algebras associated to the finite-level spaces $S_k(\Gamma_1(p^n), \cO)$. Since we have  $\cR(\rb) \cong \cT(\rb)$ by \cref{thm:BE}, the result follows.
 \end{proof}

 \begin{remark}
  Note that a crucial step in the proof of \cref{thm:BE} is to establish that the set of \emph{all} modular points (of any weight) is dense in $\fX(\rb)$. However, once this theorem is established, we can obtain the above much stronger result \emph{a posteriori}.
 \end{remark}

 For later constructions we need the fact that there exists a ``universal modular form'' over $\fX(\rb)$:

 \begin{definition} \
  \begin{enumerate}
   \item[(i)] Let $\mathbf{k}: \ZZ_p^\times \to \cR(\rb)^\times$ be the character such that $\det \rho^{\mathrm{univ}} = (\chi_{\mathrm{cyc}})^{(\mathbf{k} - 1)}$.
   \item[(ii)] Let $\cG^{[p]}_{\rb}$ be the formal power series
   \[ \cG^{[p]}_{\rb} = \sum_{p \nmid n} t_n q^n \in \cR(\rb)[[q]], \]
   where the $t_n$ are determined by the identity of formal Dirichlet series
   \[ \sum_{p \nmid n} t_n n^{-s} = \prod_{\ell \ne p} \det\left(1 - \ell^{-s} \rho^{\mathrm{univ}}(\Frob_\ell^{-1})\right)^{-1}.\]
  \end{enumerate}
 \end{definition}

 The specialisation of $\cG^{[p]}_{\rb}$ at a nearly-classical point $\rho_{f, p} \otimes (\chi_{\mathrm{cyc}})^{-t}$ is precisely the ``$p$-depletion'' $\theta^{t}(f^{[p]})$ of $\theta^{t}(f)$, where $\theta$ is the Serre--Tate differential operator $q \tfrac{\mathrm{d}}{\mathrm{d}q}$. If $t \ge 0$, this $p$-adic modular form is the image under the unit-root splitting of a classical \emph{nearly-holomorphic} cuspform, in the sense of Shimura.

 \begin{theorem}[Gouvea]
  The series $\cG_{\rb}^{[p]}$ is the $q$-expansion of a $p$-adic modular form with coefficients in $\cR(\rb)$, of tame level 1 and weight-character $\mathbf{k}$, which is a normalised eigenform for all Hecke operators.
 \end{theorem}

 \begin{proof}
  This follows readily from the duality between Hecke algebras and spaces of cusp forms.
 \end{proof}

 \subsection{The universal ordinary representation}

 The following definition is standard:

 \begin{definition}
  An \emph{ordinary refinement} of $(\rb, \bV)$ is a choice of 1-dimensional $\FF$-subspace $\bV^+ \subseteq \bV$ stable under $\rb(\Gamma_{\Qp})$, such that the inertia subgroup $I_{\Qp}$ acts trivially on $\bV^+$.
 \end{definition}

 Let us fix a choice of ordinary refinement $\bV^+$. Then there is a natural definition of ordinarity for deformations: we say that a deformation $\rho$ of $\rb$ (to some ring $A \in \CNLO$) is ordinary if $\rho|_{\Gamma_{\Qp}}$ preserves a rank one $A$-summand lifting $\bV^+$, and the action of $I_{\Qp}$ on this summand is trivial. (Note that this summand is unique if it exists, since our running hypotheses imply that $\bV/\bV^+$ is not isomorphic to $\bV^+$).

 \begin{theorem}
  Suppose $\rb$ is ordinary. Then there exists a complete local Noetherian $\cO$-algebra  representing the functor of ordinary deformations. We let $\cR^{\ord}(\rb)$ be this algebra, and $\fX^{\ord}(\rb) = \Spf \cR^{\ord}(\rb)$.
 \end{theorem}

 On the ``modular'' side, we can consider the ordinary Hecke algebra $\cT^{\ord}(\rb)$, which is the localisation at $\rb$ of the algebra of endomorphisms of $e^{\ord} \cdot \cS(1, \Zp)$ generated by all of the Hecke operators (including $U_p$). Then we have isomorphisms
 \[ \cR^{\ord}(\rb) \cong \cT^{\ord}(\rb), \]
 compatible with the isomorphisms of the previous section via the natural maps $\cR(\rb) \to \cR^{\ord}(\rb)$ and $\cT(\rb) \to \cT^{\ord}(\rb)$.

 Note that the composite $\Zp^\times \xrightarrow{\mathbf{k}} \cR(\rb) \to \cR^{\ord}(\rb)$ gives $\cR^{\ord}(\rb)$ the structure of a $\Lambda$-algebra, where $\Lambda = \cO[[\ZZ_p^\times]]$. So we have a map $\mathbf{k}: \fX^{\ord}(\rb) \to \fX_{\mathrm{cyc}} = \Spf \Lambda$.

 \begin{proposition}[Hida] \
  \begin{itemize}
   \item The ring $\cR^{\ord}(\rb)$ is finite and projective as a $\Lambda$-module, and thus has relative dimension 1 over $\cO$.
   \item If $k \ge 2$ is an integer, and $\chi: \Zp^\times \to \cO^\times$ is a Dirichlet character of conductor $p^n$, then the fibre of $\fX^{\ord}(\rb)$ at $\mathbf{k} = k + \chi$ is \'etale over $L = \operatorname{Frac} \cO$, and its geometric points biject with the normalised weight $k$ eigenforms of level $\Gamma_1(p^n)$ and character $\chi$ (if $n \ge 1$) or level $\Gamma_0(p)$ (if $n = 0$) which are ordinary and whose mod $p$ Galois representation is $\rb$.
  \end{itemize}
 \end{proposition}

 (Note that this fibre is empty if $k + \chi$ does not lie in the component of $\fX_{\mathrm{cyc}}$ determined by $\det \rb$.)

 Much as above, we can define a universal \emph{ordinary} eigenform $\cG_{\rb}^{\ord}$ with coefficients in $\cR^{\ord}(\rb)$ (whose $p$-depletion is the pullback of $\cG_{\rb}^{[p]}$ along $\fX^{\ord}(\rb) \to \fX(\rb)$, and whose $U_p$-eigenvalue is the scalar by which $\Frob_p^{-1}$ acts on $\cV^+$). However, we shall not use this explicitly here.

 More useful is the following dual construction due to Hida \cite{hida88}. The ring $\cR^{\ord}(\rb)$ has finitely many minimal primes, corresponding to irreducible components of  $\fX^{\ord}(\rb)$ (``branches''). If $\mathfrak{a}$ is a minimal prime, and we let $\cT_{\mathfrak{a}}$ be the integral closure of $\cT^{\ord}(\rb) / \mathfrak{a}$, then we can find an invertible ideal $I_{\mathfrak{a}} \triangleleft \cT_{\mathfrak{a}}$, and a homomorphism
 \[ \lambda_{\mathfrak{a}}: \mathcal{S}^{\mathrm{ord}}(1, \Lambda) \otimes_{\cT^{\ord}(\rb)} \cT_{\mathfrak{a}} \to I_{\mathfrak{a}}^{-1}, \]
 characterised by mapping $\cG_{\rb}^{\ord}$ to 1.

 \subsection{Nearly ordinary deformations}

 More generally, we can define a \emph{nearly ordinary} refinement by dropping the requirement that inertia act trivially on $\bV^+$; and there is a corresponding nearly-ordinary deformation functor, represented by a ring $\cR^{\nord}(\rb)$. If $(\bV, \bV^+)$ is nearly-ordinary, we can find a unique character $\bar\chi: \Gal(\QQ(\zeta_p) /\QQ) \to \FF^\times$ such that $(\bV \otimes \bar\chi, \bV^+ \otimes \bar\chi)$ is ordinary; and we obtain an identification of $\cR^{\nord}(\rb)$ with the tensor product of $\cR^{\ord}_{\rb \otimes \bar\chi}$ and the ring parametrising deformations of $\bar\chi$ to a character of $\Gal(\QQ(\zeta_{p^\infty}) / \QQ)$, which is isomorphic to $\cO[[X]]$. % (more canonically, it is one of the isotypical components of the Iwasawa algebra $\Lambda$).
 Thus $\cR^{\nord}(\rb)$ is flat over $\cO$ of relative dimension 2.
%
% We thus have two maps $\fX^{\nord}(\rb) \to \fX_{\mathrm{cyc}}$: a map $\mathbf{k}$, factoring through $\fX^{\mathrm{ord}}(\rb)$; and a map $\mathbf{j}$, factoring through the deformation ring of $\bar\chi$. If $k, j \in \ZZ$ with $k \ge 2$, then the fibre of $\fX^{\nord}(\rb)$ at $(\bfk, \bfj) = (k, j)$ parametrises Galois representations of the form $\rho = \rho_{f, p} \otimes (\chi_{\mathrm{cyc}})^{j}$, where $f$ is a classical weight $k$ ordinary eigenform.

 \subsection{Examples of Panchishkin families}

 The above deformation-theoretic results give rise to the following examples of Panchishkin families, in the sense of \cref{def:panchfam}.

 \begin{example}[Ordinary families of modular forms]
  \label{ex1}
  Suppose $\bV$ is a modular mod $p$ representation with a nearly-ordinary refinement $\bV^+$. Then the universal family $\cV^{\nord}$ of Galois representations over $\cR^{\nord}(\rb)$, together with its universal nearly-ordinary refinement $\cV^{\nord, +}$, is an example of a Panchishkin family. In this case, Hida theory shows that $\Sigma(\cV, \cV^+)$ consists precisely of the $\QQbar_p$ points of $\fX^{\nord}(\rb)$ of the form $\theta^{-s}(f)$, where $f$ has weight $k \ge 2$ and $1 \le s \le k-1$. These are manifestly Zariski-dense. \cref{conj:main} is known for this family, by work of Mazur and Kitagawa \cite{kitagawa94}.
 \end{example}

 We are principally interested in examples which (unlike \cref{ex1}) are \emph{not} nearly-ordinary. Our first examples of such representations come from tensor products:

 \begin{example}[Half-ordinary Rankin--Selberg convolutions]\label{ex2}
  Let $\bV_1$ and $\bV_2$ be two mod $p$ representations satisfying our running hypotheses, and suppose $\bV_1$ admits a nearly-ordinary refinement $\bV_1^+$. Twisting $\bV_1$ by a character and $\bV_2$ by the inverse of this character, we can suppose that $(\bV_1, \bV_1^+)$ is actually ordinary (not just nearly-so). Then we consider the triple $(\cR, \cV, \cV^+)$ given by
  \[
  \begin{aligned}
  \cR&= \cR^{\ord}(\rb_1) \htimes \cR(\rb_2), &\cV &= \cV_1^{\ord} \htimes \cV_2,&  \cV^+ &= \cV_1^{\ord, +} \htimes \cV_2.
  \end{aligned}
  \]
  where $(\cV_1^{\ord}, \cV_1^{\ord, +})$ is the universal ordinary deformation of $(\bV_1, \bV_1^+)$, and $\cV_2$ the universal deformation of $\bV_2$ (with no ordinarity condition). Note that $\cR$ has relative dimension 4 over $\cO$.

  The set $\Sigma(\cV, \cV^+)$ is the set of points of the form $\left(f, \theta^{-s}(g)\right)$, where $f$ is a classical point of weight $k \ge 2$, and $\theta^{-s}(g)$ is a nearly-classical point such that $g$ has weight $\ell < k$ and $s$ lies in the range of critical values of the Rankin--Selberg $L$-function, namely
  \[ \ell \le s \le k-1.\]
  This set $\Sigma(\cV, \cV^+)$ is Zariski-dense; even the specialisations with $(k, \ell, s) = (3, 2, 2)$ are dense. We shall verify \cref{conj:main} for this family below.
 \end{example}

 \begin{remark}
  A generalisation of the above two examples would be to consider tensor products of universal representations over product spaces of the form
  \[ \fX = \fX^{\nord}(\rb_1) \times \fX(\rb_2) \times \dots \times \fX(\rb_n)\]
  for general $n$, where $\rb_1, \dots, \rb_n$ are irreducible modular representations mod $p$ with $\rb_1$ nearly ordinary. This space has dimension $3n-1$; but there are $n-1$ ``redundant'' dimensions, since the tensor product is not affected by twisting $\rho_1$ by a character and one of $\rho_2, \dots, \rho_n$ by the inverse of this character. Quotienting out by this action gives a Panchishkin family over a $2n$-dimensional base.
 \end{remark}

 \begin{example}[General tensor products]
  Let $L = \operatorname{Frac} \cO$ and let $V_1$ be any $L$-linear representation of $\Gamma_{\QQ}$ (not necessarily 2-dimensional) which is geometric, satisfies the Panchishkin condition, and has $\dim V^{c = 1} = \dim V^{c = -1}$. Let $V_1^\circ$ be a $\Gamma_{\QQ}$-stable $\cO$-lattice in $V_1$ (which always exists). Then, for any modular mod $p$ representation $\bV_2$, we obtain a Panchishkin family by letting
  \[
  \begin{aligned}
  \cR&= \cR(\rb_2), &\cV &= V_1^\circ \otimes \cV_2,&  \cV^+ &= (V_1^\circ \cap V_1^+) \otimes \cV_2,
  \end{aligned}
  \]
  In particular, we can take $V_1$ to be the Galois representation arising from a cohomological automorphic representation of $\operatorname{GSp}_4$ which is Klingen-ordinary at $p$.
 \end{example}

 Note that in the last two examples the subspace $\cV^+$ will \emph{not}, in general, extend to a full flag of $\Gamma_{\Qp}$-stable subspaces, so $\cV$ is not nearly ordinary.

 \subsection{Families of Euler systems} \label{sect:nakamura}

  The canonical 2-dimensional family $\cV$ over $\cR(\rb)$ will not, in general, satisfy the Panchishkin condition. However, it automatically satisfies the more general ``$r$-Panchishkin condition'' described above if we take $r = 1$, since $\cV^+ = \{0\}$ satisfies the conditions of a 1-Panchishkin submodule (with $\Sigma(\cV, \cV^+)$ being the set of nearly-classical specialisations $\theta^t(f)$ with $t \ge 0$).

  So the more general conjecture sketched in \S \ref{sect:euler} predicts that there should exist a family of Euler systems taking values in $\cV^*(1)$, interpolating Kato's Euler systems for each modular form $f$ lifting $\rb$. Such a family of Euler systems has recently been constructed by Nakamura \cite{nakamura20}.

\section{P-adic L-functions for half-ordinary Rankin convolutions}
 \label{sect:rankin}

 Let us choose two mod $p$ representations $\rb_1$, $\rb_2$ satisfying the conditions above, with $\rb_1$ ordinary (but no ordinarity assumption on $\rb_2$).

 Choose a branch $\mathfrak{a}$ of $\fX^{\ord}(\rb_1)$ as before, and let $\mathcal{A}$ denote the ring
 \( \cT_{\mathfrak{a}} \mathop{\hat\otimes}_{\Zp} \cT(\rb_2) \), and $\fX = \fX_{\mathfrak{a}} \times \fX(\rb)$ its formal spectrum. This has relative dimension 4 over $\Zp$. We let $\mathcal{V}$ denote the $\mathcal{A}$-linear representation $\rho_{1}^{\mathrm{ord}} \otimes (\rho_2)^*(1)$, and $\mathcal{V}^+ = (\rho_{1}^{\mathrm{ord}})^+ \otimes (\rho_2)^*(1)$ where $(\rho_{1}^{\mathrm{ord}})^+$ is the 1-dimensional unramified subrepresentation of $\rho_{1}^{\mathrm{ord}} |_{\Gamma_{\QQ_p}}$. Thus $\mathcal{V}$ is a 4-dimensional family of $\Gamma_{\QQ}$-representations over $\fX$ unramified outside $p$, and $\mathcal{V}^+$ a 2-dimensional local subrepresentation of $\mathcal{V}$.

 \begin{remark}
  This differs from the $(\cV, \cV^+)$ of \cref{ex2} by an automorphism of the base ring $\cR$, so \cref{conj:main} for either one of these examples is equivalent to the other. The present setup is slightly more convenient for the proofs.
 \end{remark}
%
% \begin{definition}
%  We let $\Sigma^+$ denote the set of $\overline{\ZZ}_p$-points of $\fX$ such that $\mathcal{V}_x$ is de Rham, $\mathcal{V}^+_x$ has all Hodge--Tate weights $\ge 1$, and $\mathcal{V}_x / \mathcal{V}^+_x$ has all Hodge--Tate weights $\le 0$.
% \end{definition}

 The set $\Sigma(\cV, \cV^+)$ contains all points $(f, \theta^t(g))$ where $f$ has weight $k \ge 2$, $g$ has weight $\ell \ge 1$, and $t$ is an integer with $0 \le t \le k-\ell-1$. Our goal is to define a $p$-adic $L$-function associated to $(\mathcal{V}, \mathcal{V}^+)$, with an interpolating property at the points in $\Sigma(\cV, \cV^+)$.

 The ring $\mathcal{A}$ is endowed with two canonical characters $\mathbf{k}_1, \mathbf{k}_2: \Zp^\times \to \mathcal{A}^\times$, the former factoring through $\cT_{\mathfrak{a}}$ and the latter through $\cT(\rb_2)$. We can regard $\cG_{\rb_2}^{[p]}$ as a $p$-adic eigenform with coefficients in $\mathcal{A}$, of weight $\mathbf{k}_2$, by base extension.

 \begin{definition}
  Let $\Xi$ denote the $p$-adic modular form
  \[ e^{\mathrm{ord}} \left(\cG_{\rb_2}^{[p]} \cdot \cE_{\mathbf{k}_1 - \mathbf{k}_2}^{[p]}\right) \in \mathcal{S}_{\mathbf{k}_1}^{\ord}(1, \mathcal{A}), \]
  where $\cE_{\mathbf{k}}^{[p]} = \sum_{\substack{n \ge 1 \\ p \nmid n}} (\sum_{d \mid n} d^{\mathbf{k} - 1}) q^n \in \mathcal{S}_{\mathbf{k}}(1, \Lambda)$ denotes the $p$-depleted Eisenstein series of weight $\mathbf{k}$ and tame level 1. Let
  \[ \mathcal{L} \coloneqq \lambda_{\mathfrak{a}}\left(\Xi\right) \in I_{\mathfrak{a}}^{-1} \otimes_{\cT_{\mathfrak{a}}}\mathcal{A}.\]
 \end{definition}

 This is a meromorphic formal-analytic function on the 4-dimensional space $\fX_{\mathfrak{a}} \times \fX(\rb)$, regular along any 3-dimensional slice $\{f\} \times \fX(\rb)$ with $f$ classical.

 We now show that the values of $\mathcal{L}$ at points in $\Sigma^+$ interpolate values of Rankin $L$-functions. Let $(f, \theta^t(g))$ be such a point, with $f, g$ newforms of $p$-power levels, and let $k, \ell$ be the weights of $f, g$. Let $\alpha$ be the eigenvalue of geometric Frobenius on the unramified subrepresentation of $\rho_{f, p} |_{\Gamma_{\QQ_p}}$, and let $f_\alpha$ be the $p$-stabilisation of $f$ of $U_p$-eigenvalue $\alpha$.

 \begin{remark}
  If $f$ has non-trivial level, then $f_\alpha = f$, and $\alpha$ is just the $U_p$-eigenvalue of $f$. If $f$ has level one, then $\alpha$ is the unique unit root of the polynomial $X^2 - a_p(f) X + p^{k-1}$, and $f_\alpha$ is the level $p$ eigenvector $f_\alpha(\tau) = f(\tau) - \frac{p^{k-1}}{\alpha} f(p\tau)$.
 \end{remark}

 We define $\lambda_{f, \alpha}$ to be the unique linear functional on $\mathcal{S}_k^{\mathrm{ord}}(1, L)$ which factors through projection to the $f_\alpha$ eigenspace, and satisfies $\lambda_{f, \alpha}(f_\alpha) = 1$. By definition, we have
 \[
 \mathcal{L}_{\mathfrak{a}}(\rb_1, \rb_2) (f, \theta^t(g)) =\lambda_{f, \alpha}\left(\theta^t(g^{[p]})\cdot E^{[p]}_{k - \ell - 2t}\right).
 \]

 \begin{definition}
  For $f$, $g$ newforms as above, we write $L^{(p)}(f \times g, s)$ for the Rankin--Selberg $L$-function of $f$ and $g$ without its Euler factor at $p$,
  \begin{align*}
  L^{(p)}(f \times g, s) &:= L^{(p)}(\chi_f\chi_g, 2s+2-k-\ell) \sum_{\substack{n \ge 1 \\ p\nmid n}} a_n(f) a_n(g) n^{-s}\\
  &= \prod_{\ell \ne p} \det\left( 1 - \ell^{-s} \Frob_\ell^{-1} : V_p(f) \otimes V_p(g)\right)^{-1},
  \end{align*}
  and let
  \[ \Lambda^{(p)}(f \otimes g, s) \coloneqq \Gamma_{\mathbb{C}}(s) \Gamma_{\mathbb{C}}(s - \ell + 1) L^{(p)}(f \otimes g, s).\]
 \end{definition}

 \begin{theorem}
  We have
  \[ \mathcal{L}(f, \theta^t(g)) =
  2^{1-k} (-1)^{t}i^{k+\ell} \left( \frac{p^{(t + 1)}}{\alpha} \right)^b \lambda_{p^b}(g)\frac{P_p(g, p^t \alpha^{-1})}
  {P_p(g^*, p^{-(\ell + t)} \alpha)} \frac{\Lambda^{(p)}(f, g^*, \ell + t)}{\mathcal{E}_p^{\mathrm{ad}}(f) \langle f, f \rangle}, \]
  where $b$ is the level at which $g$ is new. Here $P_p(g, X)$ is the polynomial such that
  \[ P_p(g, X)^{-1} = \sum_{r \ge 0} a_{p^r}(g) X^r, \]
  and
   \[
   \mathcal{E}_p^{\mathrm{ad}}(f) =
   \begin{cases}
   \left(1 - \frac{p^{k-1}}{\alpha^2}\right)\left(1 - \frac{p^{k-2}}{\alpha^2}\right) & \text{$f$ crystalline at $p$}, \\[2mm]
   -\left( \frac{p^{k-1}}{\alpha^2} \right)  & \text{$f$ semistable non-crystalline at $p$},\\[2mm]
   \left( \frac{p^{k-1}}{\alpha^2} \right)^a G(\chi_f) & \text{$f$ non-semistable at $p$, new of level $p^a$.}
   \end{cases}
   \]
 \end{theorem}

 \begin{proof}
  This follows from the Rankin--Selberg integral formula. The computations are virtually identical to the case of finite-slope forms treated in \cite{loeffler18}, so we shall not reproduce the computations in detail here.
 \end{proof}

 \begin{remark}
  Note that the factor $\frac{P_p(g, p^t \alpha^{-1})}
  {P_p(g^*, p^{-(\ell + t)} \alpha)}$ can be written as
  \[ \det\left[ (1 - \varphi)^{-1}(1 - p^{-1} \varphi^{-1}) : \mathbf{D}_{\mathrm{cris}}(V^+)\right]\]
  where $V^+ = (\rho_{f, p})^+ \otimes \rho_{g, p}^*(1+t)$ is the fibre of $\mathcal{V}^+$ at $(f, \theta^t(g))$. On the other hand, the factor $\left( \frac{p^{(t + 1)}}{\alpha} \right)^b \lambda_{p^b}(g)$ is essentially the local $\varepsilon$-factor of this representation.
 \end{remark}

\section{Other cases}

 We briefly comment on some other cases which can be treated by the same methods as above.

 \subsection{Relaxing the tame levels}

 Firstly, the assumption that the levels of our families be 1 should be easy to remove; the only price that must be paid is a little more careful book-keeping about the local Euler factors at the bad primes.

 \subsection{The case of GSp(4) \texorpdfstring{$\times$}{x} GL(2)}
  \label{sect:gsp4gl2}
  A more ambitious case which can be treated by the same methods is the following. Let $\Pi$ be a cohomological automorphic representation of $\operatorname{GSp}_4$ which is globally generic, unramified and Klingen-ordinary at $p$, and contributes to cohomology with coefficients in the algebraic representation of weight $(r_1, r_2)$, for some $r_1 \ge r_2 \ge 0$. (Classically, these correspond to holomorphic vector-valued Siegel modular forms taking values in the representation $\operatorname{Sym}^{r_1 - r_2} \otimes \det^{r_2 + 3}$ of $\GL_2$.) For technical reasons we assume $r_2 > 0$.

  In \cite{LPSZ} we constructed a cyclotomic $p$-adic $L$-function interpolating the critical values of $L(\Pi \otimes \sigma, s)$ where $\sigma$ is an automorphic representation of $\GL_2$ generated by a holomorphic form of weight $\ell \le r_1 - r_2 + 1$. This is constructed by applying a ``push-forward'' map to the product of the $p$-depleted newform $g^{[p]} \in \sigma$ with an auxiliary $p$-adic Eisenstein series, and pairing this with a coherent $H^2$ eigenclass coming from $\Pi$.

  This construction is closely parallel to the construction of the $p$-adic Rankin--Selberg $L$-function for $\GL_2 \times \GL_2$, and it generalises to universal-deformation families in the same way, since the pushforward map of \cite{LPSZ} can be applied to any family of $p$-adic modular forms (over any base). If  we assume for simplicity that $\Pi$ is unramified at all finite places, and replace $g$ with a universal deformation family $\cG^{[p]}_{\rb}$ as above, then we obtain an element of $\cR(\rb)$ interpolating these $p$-adic $L$-functions, with $\Pi$ fixed and $\sigma$ varying through the small-weight specialisations of a 3-dimensional universal-deformation family. We can also add a fourth variable, in which we vary $\Pi$ through a 1-dimensional family of Klingen-ordinary representations, with $r_1$ varying but $r_2$ fixed.

 \subsection{Self-dual triple products}

  If we are given three mod $p$ modular representations $\rho_1, \rho_2, \rho_3$ with $\rho_1$ nearly-ordinary and $\det(\rho_1) \cdot \det(\rho_2) \cdot \det(\rho_3) = \bar{\chi}_{\mathrm{cyc}}$, then the space
  \[ \left\{ (\rho_1, \rho_2, \rho_3) \in \fX^{\nord}(\rb_1) \times \fX(\rb_2) \times \fX(\rb_3) : \det(\rho_{1}) \cdot \det(\rho_{2}) \cdot \det(\rho_{3}) = \chi_{\mathrm{cyc}} \right\} \]
  carries a natural 8-dimensional Panchishkin family $\cV$, given by the tensor product of the three universal deformations $\cV_i$, with the Panchishkin submodule given by $\cV_1^+ \otimes \cV_2 \otimes \cV_3$. The base space is \emph{a priori} 7-dimensional, but it has two ``redundant'' dimensions (since we can twist either $\rho_2$ or $\rho_3$ by a character, and $\rho_1$ by the inverse of that character, without changing the tensor product representation), so we obtain a Panchishkin family over a 5-dimensional base $\fX$, satisfying the self-duality condition $\cV \cong \cV^*(1)$. The set $\Sigma(\cV, \cV^+)$ corresponds to triples of classical modular forms $(f_1, f_2, f_3)$ which are ``$f_1$-dominant'' -- i.e.~their weights $(k_1, k_2, k_3)$ satisfy $k_1 \ge k_2 + k_3$.

  Feeding the universal eigenforms $\cG^{[p]}_{\rb_2}$ and $\cG^{[p]}_{\rb_3}$ into the construction of \cite{DR14} gives a $p$-adic $L$-function over this 5-dimensional base space , extending the construction in \emph{op.cit.}~of a $p$-adic $L$-function over the 3-dimensional subspace of $\fX$ where $\rho_2$ and $\rho_3$ are nearly-ordinary.

  (Note that this is actually a refinement of \cref{conj:main}, since the resulting $p$-adic $L$-function interpolates the square-roots of central $L$-values.)

 \subsection{The Bertolini--Darmon--Prasanna case}

  Let $\rb$ be a modular mod $p$ representation of $\Gamma_{\QQ, \{p\}}$, with universal deformation space $\fX(\rb)$. We shall suppose that $\det \rb = \bar{\chi}_{\mathrm{cyc}}$, and we let $\fX^{0}(\rb) \subseteq \fX(\rb)$ denote the subspace parametrising deformations whose determinant is $\chi_{\mathrm{cyc}}$; this is flat over $\cO$ of relative dimension 2, and is formally smooth if $\rb$ is unobstructed.

  Meanwhile, we choose an imaginary quadratic field $K$ in which $p = \mathfrak{p}_1 \mathfrak{p}_2$ is split, and we let $\fX_K^{\mathrm{ac}} \cong \operatorname{Spf} \cO[[X]]$ be the character space of the anticyclotomic $\Zp$-extension of $K$.

  Let $\fX$ denote the product $\fX^{\mathrm{ac}}_K \times \fX^{0}(\rb)$. This is $\cO$-flat of relative dimension 3, and it carries a family of 4-dimensional Galois representations $\cV$, given by tensoring the universal deformation $\rho^{\mathrm{univ}}$ of $\rb$ with the induction to $\Gamma_{\QQ}$ of the universal deformation of $\bar\psi$. Note that $\cV$ satisfies the ``self-duality'' condition $\cV^\vee(1) \cong \cV$. Locally at $p$, $\cV$ is the direct sum of two twists of the universal deformation of $\rb$, corresponding to the two primes above $p$; and we can define a Panchishkin submodule $\cV^+$ by taking the direct summand corresponding to one of these primes. Note that $\Sigma(\cV, \cV^+)$ consists of pairs $(\psi, f)$ where $f$ is a modular form and $\psi$ an anticyclotomic algebraic Hecke character of weight $(n, -n)$, where $n$ is large compared to the weight of $f$.

  Plugging in the universal family $\cG^{[p]}_{\rb}$ (more precisely, its pullback to $\fX^{0}(\rb)$) into the constructions of \cite{BDP13}, we obtain a $p$-adic analytic function on the 3-dimensional space $\fX^{\mathrm{ac}} \times \fX^{0}(\rb)$ interpolating the square-roots of central $L$-values at specialisations in $\Sigma(\cV, \cV^+)$. This refines the construction due to Castella \cite[\S 2]{castella19} of a BDP-type $L$-function over the 2-dimensional space $\fX^{\mathrm{ac}}_K \times \fX^{\ord}(\rb)$ when $\rb$ is ordinary.\footnote{This is slightly imprecise since $\fX^{\ord}(\rb)$ is not contained in $\fX^{0}(\rb)$; more precisely, the correspondence between the two constructions is given by identifying $\fX^{\ord}(\rb)$ with $\fX^{\nord}(\rb) \cap \fX^{0}(\rb)$, via twisting by a suitable character of $\Gamma_{\QQ, \{p\}}^{\mathrm{ab}}$.}

 \subsection{A finite-slope analogue?}

  One can easily formulate a ``finite-slope'' analogue of \cref{conj:main}, where the submodule $\cV^+ \subseteq \cV$ is replaced by a submodule of the Robba-ring $(\varphi, \Gamma)$-module of $\cV|_{\Gamma_{\Qp}}$. The analogue of Hida's ordinary deformation space $\fX^{\ord}(\rb)$ is now the $\rb$-isotypic component $\cE(\rb)$ of the Coleman--Mazur Eigencurve \cite{CMeigen}.

  However, proving a finite-slope version of the results of \cref{sect:rankin}, or of the generalisations sketched in the above paragraphs, appears to be much more difficult than the ordinary case. All of the above constructions rely on the existence of the universal eigenform $\cG^{[p]}_{\rb}$ as a family of $p$-adic modular forms over $\fX(\rb)$. However, in the finite-slope case, we need to pay attention to overconvergence conditions, since the finite-slope analogue of the projectors $\lambda_{\mathfrak{a}}$ are only defined on overconvergent spaces. Clearly $\cG^{[p]}_{\rb}$ is not overconvergent (as a family), since it has specialisations which are nearly-classical rather than classical. So we need to work in an appropriate theory of nearly-overconvergent families. Such a theory has recently been introduced by Andreatta and Iovita \cite{AI17}. We might make the following optimistic conjecture:

  \begin{conjecture}
   Let $f$ be a nearly-classical point of $\fX(\rb)$, corresponding to a modular form $f$ of prime-to-$p$ level. Then there is an affinoid neighbourhood $X_f = \operatorname{Max} A_f$ of $f$ in $\fX(\rb)^{\mathrm{an}}$ over which the universal eigenform $\cG_{\rb}^{[p]}$ is a family of nearly-overconvergent forms in the sense of \cite{AI17}.
  \end{conjecture}

  If this conjecture holds, one might realistically hope to define (for instance) a $p$-adic Rankin--Selberg $L$-function over neighbourhoods of crystalline classical points in $\cE(\rb_1) \times \fX(\rb_2)^{\mathrm{an}}$.

\section{Conjectures on $P$-nearly-ordinary families}

 In this section, we'll use Galois deformation theory to define universal parameter spaces for Galois representations valued in reductive groups, which satisfy a Panchishkin-type condition relative to a parabolic subgroup; and we formulate a ``parabolic $\cR = \cT$'' conjecture, predicting that these should have an alternative, purely automorphic description. We expect that these parameter spaces should be the natural base spaces for families of $p$-adic $L$-functions, and of Euler systems.

 \subsection{P-nearly-ordinary deformations}

 Let $G$ be a reductive group scheme over $\cO$ and $P$ a parabolic subgroup. In \cite[\S 7]{boeckle07}, B\"ockle defines a homomorphism $\rho: \Gamma_{\QQ, S} \to G(A)$, for $A \in \CNLO$, to be \emph{$P$-nearly ordinary} if $\rho|_{\Gamma_{\Qp}}$ lands in a conjugate of $P(A)$. Theorem 7.6 of \emph{op.cit.} shows that under some mild hypotheses, the functor of $P$-nearly-ordinary deformations of a given $P$-nearly-ordinary residual representation is representable.

 The notion of a \emph{Panchishkin family} introduced in Definition \ref{def:panchfam} corresponds to taking $G = \GL_n$ and $P$ to be the parabolic subgroup with blocks of sizes $\dim \bV^{c=1}$ and $\dim \bV^{c=-1}$. However, the geometry of deformation spaces for $\GL_n$ is rather mysterious when $n > 2$, and it is not expected that these spaces will have a Zariski-dense set of classical points. On the other hand, the geometry of deformation spaces is much simpler and better-understood for Galois representations arising from Shimura varieties (or, more generally, from automorphic representations that are discrete-series at $\infty$).

 This suggests concentrating on the following setting. Let $G$ be a reductive group over $\QQ$; for simplicity, we assume here $G$ is split. We also suppose $G$ has a ``twisting element'' in the sense of \cite{buzzardgee}, and fix a choice of such an element\footnote{Alternatively, one could replace $G^\vee$ by the connected component of the ``$C$-group'' of \emph{op.cit.}, which the quotient of $G^\vee \times \mathbf{G}_m$ by a central element of order 2. We can also allow non-split $G$, by considering representations into a larger, non-connected quotient of the $C$-group.}. Then Conjecture 5.3.4 of \emph{op.cit.} predicts that cohomological automorphic representations $\Pi$ of $G$ give rise to Galois representations $\rho_{\Pi, p}: \Gamma_{\QQ} \to G^\vee(\QQbar_p)$, where $G^\vee$ is the Langlands dual of $G$.

 There is a canonical bijection $P \leftrightarrow P^\vee$ between conjugacy classes of parabolics in $P$ and parabolics in $G^\vee$, and one expects that if $\Pi$ is nearly-ordinary for $P$ (in the sense that the Hecke operators associated to $P$ have unit eigenvalues), then $\rho_{\Pi, p}$ should be a $P^\vee$-nearly-ordinary representation. In particular, families of $P$-nearly-ordinary cohomological automorphic representations of $G$ should give rise to families of $P^\vee$-nearly-ordinary Galois representations into $G^\vee$.

 If we also choose a linear representation $\xi: G^\vee \to \GL_n$, then for suitably chosen $P$, the resulting families of $n$-dimensional Galois representations will be Panchishkin families. The example of \S\ref{sect:rankin} is of this type, taking $G = G' = \GL_2 \times \GL_2$, and $P = P' = B_2 \times \GL_2$ where $B_2$ is the Borel subgroup of $\GL_2$, and $\xi$ the 4-dimensional tensor product representation of $G$. Similarly, the self-dual triple-product setting of Example \ref{ex2} corresponds to taking $G = (\GL_2 \times \GL_2 \times \GL_2) / \GL_1$, and $P$ the image of $B_2 \times \GL_2 \times \GL_2$.

 \subsection{Big and small Galois eigenvarieties}

  In the above setting, we define the \emph{big $P$-nearly-ordinary Galois eigenvariety} for $G$ to be the following space. Suppose $G^\vee$ and $P^\vee$ have smooth models over $\cO$, and fix some choice of $\rb: \Gamma_{\QQ, S} \to G^\vee(\FF)$ which is $P^\vee$-nearly-ordinary. Then  -- assuming the hypotheses of B\"ockle's construction are satisfied -- we obtain a universal deformation ring $\cR^{P^\vee-\nord}(\rb)$ for for $P^\vee$-nearly-ordinary liftings of $\rb$. We define the big $P$-nearly-ordinary Galois eigenvariety $\fX_P(\rb)$ to be the formal spectrum of this ring $\cR^{P^\vee-\nord}(\rb)$.

  The methods of \cite{boeckle07} give a formula for the dimension of this space. Suppose $\rb$ satisfies the ``oddness'' condition that $\dim \gf^{\rb(c) = 1} = \dim(G / B_G)$, where $\gf$ is the Lie algebra of $G / \FF$, $c$ is complex conjugation and $B_G$ is a Borel subgroup of $G$. (This condition is expected to hold for representations arising from Shimura varieties; see \cite[Introduction]{CHT08}.) Then $\cR^{P^\vee-\nord}(\rb)$ has a presentation as a quotient of a power series ring in $d_1$ variables by an ideal with $d_2$ generators, where
  \[ d_1 - d_2 = \dim P - \dim(G/B_G) = \dim B_M,\]
  where $M$ is the Levi factor of $P$ and $B_M \subseteq M$ is a Borel subgroup of $M$. It seems reasonable to conjecture that $\fX_P(\rb)$ is in fact flat over $\cO$, and its relative dimension is $\dim B_M$.

  The term \emph{big} is intended to contrast with the following alternative construction (which is perhaps less immediately natural; we introduce it because it is the Galois counterpart of an existing construction on the automorphic side, as we shall recall below). Let $\overline{M^\vee} = M^\vee / Z(M^\vee)$ (the Langlands dual of $M^{\der}$), and fix a \emph{Hodge type} $\mathbf{v}$ and an \emph{inertial type} $\tau$ for $\overline{M^\vee}$-valued representations of $\Gamma_{\Qp}$, in the sense of \cite{bellovingee19}. Then we say a lifting $\rho$ of $\rb$ to $\QQbar_p$ is \emph{$P^\vee$-nearly-ordinary of type $(\tau, \mathbf{v})$} if it is $P^\vee$-nearly-ordinary, and the composition $\Gamma_{\Qp} \xrightarrow{\rho} P^\vee(\QQbar_p) \to \overline{M^\vee}(\QQbar_p)$ has the given Hodge and inertial types. We define the \emph{small $P$-nearly-ordinary Galois eigenvariety} to be the universal deformation space $\fX_P(\rb; \tau, \mathbf{v})$ for deformations that are $P^\vee$-nearly-ordinary of the specified type. Using the formulae of \cite{bellovingee19} applied to $\overline{M^\vee}$ to compute the dimension of the local lifting rings, and assuming that $\rb$ is odd and $\mathbf{v}$ is sufficiently regular, we compute that the expected dimension of $\fX_P(\rb; \tau, \mathbf{v})$ is now given by $\dim Z_{M^\vee} = \dim Z_{M}$.

  \begin{remark}
   Note that the big and small Galois eigenvarieties coincide if $P$ is a Borel subgroup; but the dimension of the big eigenvariety \emph{grows} with $P$, while the dimension of the small eigenvariety \emph{shrinks} as $P$ grows. For instance, if $G = \GL_2$ and $P = G$, then $\fX_P(\rb)$ is just the unrestricted deformation space, which is 3-dimensional over $\cO$ as we have seen; but $\fX_P(\rb; \tau, \mathbf{v})$ has dimension 1, since for any $(\tau, \mathbf{v})$ there are only finitely many deformations of that type, so $\fX_P(\rb; \tau, \mathbf{v})$ has only finitely many points up to twisting by characters.
  \end{remark}

 \subsection{Big and small automorphic eigenvarieties}

  We can now ask if the above Galois-theoretic spaces have automorphic counterparts.

  \subsubsection{The big eigenvariety} Seeking an automorphic counterpart of the big Galois eigenvariety leads to the following question:\medskip

  \noindent\textbf{Question}. If $G$ is reductive over $\QQ$, and $P$ is a parabolic in $G / \Qp$ as above, is there a natural purely automorphic construction of a parameter space $\mathfrak{E}_P$ for systems of Hecke eigenvalues arising from cohomological automorphic representations for $G$ that are nearly ordinary for the parabolic $P$? \medskip

  We call this conjectural object $\mathfrak{E}_P$ the \emph{big $P$-nearly-ordinary automorphic eigenvariety}. We expect its dimension to be the same as its Galois analogue; in particular, if $G$ has discrete series its dimension should be $\dim B_M$, where $B_M$ is a Borel subgroup of the Levi of $P$ as before.

  The case when $P = B$ is a Borel subgroup is relatively well-understood; this is the setting of Hida theory. However, the case of non-Borel parabolics is much more mysterious. In this case, one can give a candidate for this space $\mathfrak{E}_P$ as follows.

  For any open compact $K \subset G(\mathbf{A}_{\mathrm{f}})$, we can form the $H^*(K, \cO)$ of Betti cohomology of the symmetric space for $G$ of level $K$, which is a finitely-generated graded $\cO$-module. This has an action of Hecke operators, and the subalgebra of its endomorphisms generated by Hecke operators at primes where $K$ is unramified, the \emph{spherical Hecke algebra of level $K$}, is commutative.

  We fix an open compact subgroup $K^p \subset G(\mathbf{A}_\mathrm{f}^p)$, and let $K_{n, p} = \{ g \in G(\Zp): g \bmod p^n \in N_P(\ZZ/p^n)\}$, where $N_P$ is the unipotent radical of $P$. Then, for any $n \ge 1$, $H^*(K, \cO)$ has a canonical idempotent endomorphism $e_P$ (the Hida ordinary projector associated to $P$), defined by $\lim_{r \to \infty} U_P^{r!}$ where $U_P$ is a suitable Hecke operator; this commutes with the spherical Hecke algebra.

  \begin{definition}
   With the above notations, let $\cT^{P-\nord}_n(K^p)$ be the quotient of the spherical Hecke algebra acting faithfully on $e_P H^*(K^p K_{p, n}, \cO)$; and define $\cT^{P-\nord}(K^p) = \varprojlim_n \cT^{P-\nord}_n(K^p)$.
  \end{definition}

  We conjecture that the formal spectrum of $\cT^{P-\nord}(K^p)$ should be the big $P$-nearly-ordinary eigenvariety. However, from this definition alone it is rather difficult to obtain much information about the properties of the resulting space (for instance, it is not clear whether $\cT^{P-\nord}(K^p)$ is Noetherian). As far as the author is aware, the only non-Borel cases where this construction is well-understood are the following:
  \begin{itemize}
  \item $G = \GL_2$ and $P = G$, as in Theorem \ref{thm:BE}.
  \item $G = \operatorname{Res}_{F^+ / \QQ}(U)$, where $U$ is a totally definite unitary group for some CM extension $F / F^+$, with $p$ split in $F$ and $F / F^+$ unramified at all finite places; and $P$ is a parabolic subgroup of $G(\Qp) \cong \GL_n(\Qp)^{[F^+: \Qp]}$ whose Levi subgroup is a product of copies of $\GL_1$ and $\GL_2$. This case has been studied extensively by Yiwen Ding \cite{ding19}.
  \end{itemize}

  In the definite unitary case, Ding proves that the localisation of $\cT^{P-\nord}(K^p)$ at the maximal ideal corresponding to an irreducible $\rb$ is a quotient of the global Galois deformation ring $\cR^{P^\vee-\nord}(\rb)$, and is therefore Noetherian; and he gives a lower bound for the relative dimension of $\cT^{P-\nord}(K^p)$ over $\cO$ (localised at the maximal ideal corresponding to some $\bar{\rho}$). This lower bound is exactly $\dim B_M$, the dimension conjectured for the Galois eigenvariety above.

  \begin{remark}
   Note that Ding's construction uses the $p$-adic local Langlands correspondence for $\GL_2(\Qp)$ in an essential way, so this approach will be much harder to generalise to cases where the Levi of $P$ is not a product of tori and copies of $\GL_2(\Qp)$.
  \end{remark}

  \subsubsection{The small eigenvariety} In contrast to the rather disappointing situation described above, there does seem to be a well-established theory for the ``little brother'' of this space -- the \emph{small $P$-nearly-ordinary automorphic eigenvariety}. This would be a parameter space for $P$-nearly-ordinary cohomological automorphic representations satisfying two additional conditions:
  \begin{itemize}
   \item the highest weight $\lambda$ of the algebraic representation of $G$ to whose cohomology $\Pi$ contributes should lie in a fixed equivalence class modulo characters of $M / M^{\der}$;
   \item the ordinary part $J_P(\Pi_p)^{\nord}$ of $J_P(\Pi_p)$, which is an irreducible smooth representation of $M(\Qp)$, should satisfy $e \cdot J_P(\Pi_p)^{\nord} \ne 0$ where $e$ is some fixed idempotent in the Hecke algebra of $M^{\der}(\Qp)$.
  \end{itemize}
  Note that both conditions are vacuous if $P$ is a Borel. These conditions are the automorphic counterparts of the fixed Hodge and inertial types up to twisting used to define the small $P$-nearly-ordinary Galois eigenvariety. See e.g. Mauger \cite{mauger05} for the construction of the small $P$-nearly-ordinary automorphic eigenvariety, and \cite{hillloeffler11} for a ``$P$-finite-slope'' analogue.

  \begin{remark}
   The most obvious choice of $e$ would be the idempotent projecting to the invariants for some choice of open compact subgroup of $M^{\der}(\Qp)$. For instance, Mauger's theory applies to $\Pi$ such that $J_P(\Pi_p)^{\nord}$ has non-zero invariants under $M^{\der}(\Zp)$, although it can be extended without difficulty to allow other more general idempotents. However, a craftier choice would be to take $e$ to be a \emph{special idempotent} in the sense of \cite{bushnellkutzko98}, corresponding to a choice of Bernstein component for $M^{\der}(\Qp)$; these Bernstein components are expected to biject with inertial types on the Galois side (the inertial local Langlands correspondence for $M^{\der}(\Qp)$), while the highest weights $\lambda$ biject with Hodge types, so we obtain a natural dictionary between the defining data at $p$ for the Galois and automorphic versions of the small $P$-nearly-ordinary eigenvariety.
  \end{remark}

  \subsubsection{$R = T$ theorems} Both big and small automorphic eigenvarieties should, clearly, decompose into disjoint unions of pieces indexed by mod $p$ Hecke eigenvalue systems. We can then formulate the (\emph{extremely} speculative) ``parabolic $R = T$'' conjecture that each of these pieces should correspond to one of the big or small Galois eigenvarieties of the previous section, for a mod $p$ Galois representation $\rb$ determined by the mod $p$ Hecke eigensystem.

  In the case when $G$ is a definite unitary group, results of this kind have been proven by Geraghty \cite{geraghty} when $P$ is a Borel subgroup; and when the Levi of $P$ is a product of $\GL_1$'s and $\GL_2$'s, Ding proves in \cite{ding19} the slightly weaker result that the map from $\cR^{P^\vee-\nord}(\rb)$ to the $\rb$-localisation of $\cT^{P-\nord}(K^p)$ is surjective with nilpotent kernel, after possibly extending the totally real field $F^+$ (an ``$R^{\mathrm{red}} = T^{\mathrm{red}}$'' theorem).

  \subsection{Miscellaneous remarks}
  \begin{remark}
   The 4-dimensional parameter space for $\operatorname{GSp}_4 \times \GL_2$ mentioned at the end of \S\ref{sect:gsp4gl2} is a slightly artificial hybrid: the it is the product of the \emph{big} automorphic (or Galois) eigenvariety for $P = G = \GL_2$ with the \emph{small} automorphic eigenvariety for the Klingen parabolic of $\operatorname{GSp}_4$. Of course, we expect that the ``correct'' parameter space for this construction is the product of the big eigenvarieties for the two groups, which would have dimension 7 (or 6 if we factor out a redundant twist, which corresponds to working with the group $\operatorname{GSp}_4 \times_{\GL_1} \GL_2$). However, we do not know how to construct $p$-adic $L$-functions on this eigenvariety at present.
  \end{remark}

  \begin{remark}
   The small $P$-nearly-ordinary eigenvariety is finite over the ``weight space'' parametrising characters of $(M / M^{\der})(\Zp)$. Moreover, in Shimura-variety settings it is flat over this space (up to a minor grain of salt if $Z_G$ has infinite arithmetic subgroups). It is natural to ask if there is an analogous, purely locally defined ``big $P$-weight space'' over which the big eigenvariety $\mathfrak{E}_P$ is finite; the results of \cite{ding19} suggest that a candidate could be a universal deformation space for $p$-adic Banach representations of $M(\Qp)$ on the automorphic side, or $M^\vee$-valued representations of $\Gamma_{\Qp}$ on the Galois side. However, these spaces will in general have much larger dimension than the eigenvariety, so there does not seem to be a natural choice of local parameter space over which $\mathfrak{E}_P$ is finite and flat.
  \end{remark}

\providecommand{\bysame}{\leavevmode\hbox to3em{\hrulefill}\thinspace}
\providecommand{\MR}[1]{}
\renewcommand{\MR}[1]{%
 MR \href{http://www.ams.org/mathscinet-getitem?mr=#1}{#1}.
}
\providecommand{\href}[2]{#2}
\newcommand{\articlehref}[2]{\href{#1}{#2}}

\end{document}